\documentclass[reqno]{amsart}
\usepackage{amssymb}
\usepackage[mathscr]{eucal}
\usepackage{hyperref}
\usepackage[retainorgcmds]{IEEEtrantools}
\usepackage{graphicx}
\usepackage{xcolor}
\usepackage[all]{xy}
\usepackage{tikz}
\newcommand\nos{1000}

\newcommand{\weglassen}[1]{}
\usetikzlibrary{shapes,patterns,calc,snakes}

\theoremstyle{plain}
\newtheorem{prop}{Proposition}[section]
\newtheorem{thm}[prop]{Theorem}
\newtheorem{cor}[prop]{Corollary}
\newtheorem{lem}[prop]{Lemma}

\theoremstyle{definition}
\newtheorem{dfn}[prop]{Definition}
\newtheorem{rem}[prop]{Remark}

\newtheorem{example}[prop]{Example}
\newtheorem{examples}[prop]{Examples}

\newtheorem{lab}[prop]{}

\newcommand{\into}{\hookrightarrow}

\newcommand{\To}{\Rightarrow}

\newcommand{\ratto}{\dashrightarrow}
\renewcommand{\iff}{\Leftrightarrow}
\renewcommand{\subset}{\subseteq}

\newcommand{\A}{{\mathbb{A}}}

\newcommand{\G}{{\mathbb{G}}}
\newcommand{\N}{{\mathbb{N}}}
\renewcommand{\P}{{\mathbb{P}}}

\newcommand{\R}{{\mathbb{R}}}
\newcommand{\Z}{{\mathbb{Z}}}

\newcommand{\scrO}{{\mathscr{O}}}
\newcommand{\scrP}{{\mathscr{P}}}

\newcommand{\scrV}{{\mathscr{V}}}

\newcommand{\x}{{\mathtt{x}}}

\DeclareMathOperator{\Hom}{Hom}

\DeclareMathOperator{\Pic}{Pic}

\DeclareMathOperator{\rk}{rk}
\DeclareMathOperator{\Spec}{Spec}

\DeclareMathOperator{\supp}{supp}

\DeclareMathOperator{\trdeg}{trdeg}

\newcommand{\cone}{\mathrm{cone}}

\newcommand{\inter}{\mathrm{int}}

\newcommand{\relint}{\mathrm{relint}}

\newcommand{\lf}{\mathrm{in}}
\newcommand{\interior}{\mathrm{interior}}
\newcommand{\ord}{\mathrm{ord}}

\renewcommand{\emptyset}{\varnothing}
\renewcommand{\setminus}{\smallsetminus}
\renewcommand{\epsilon}{\varepsilon}

\newcommand{\mal}{\,.\,}
\newcommand{\ol}{\overline}
\newcommand{\plus}{{\scriptscriptstyle+}}

\newcommand{\wt}[1]{\widetilde{#1}}
\newcommand{\all}{\forall\,}

\newcommand{\bil}[2]{\langle{#1},{#2}\rangle}

\newcommand{\marginnote}[1]{\vrule width0pt height0pt depth0pt
  \vadjust{\vbox to0pt{\vss\hbox to\hsize{\hskip\hsize\quad
  #1\hss}\vskip1.5pt}}}

\newcommand{\sa}{semi-algebraic}

\begin{document}

\title
[Toric completions and bounded functions]
{Toric completions and bounded functions\\ on real algebraic varieties}

\author{Daniel Plaumann}
\author{Claus Scheiderer}


\address{Fakult\"at f\"ur Mathematik, Technische Universit\"at
  Dortmund, Vogelpothsweg 87, 44227 Dortmund, Germany}
\address{Fachbereich Mathematik und Statistik, Universit\"at Konstanz,
  78457 Konstanz, Germany}

\email{Daniel.Plaumann@math.tu-dortmund.de}
\email{Claus.Scheiderer@uni-konstanz.de}

\date\today
\maketitle


\begin{abstract}
  Given a semi-algebraic set $S$, we study compactifications of $S$
  that arise from embeddings into complete toric varieties. This
  makes it possible to describe the asymptotic growth of polynomial
  functions on $S$ in terms of combinatorial data. We extend our
  earlier work in \cite{ps} to compute the ring of bounded functions
  in this setting and discuss applications to positive polynomials
  and the moment problem. Complete results are obtained in special
  cases, like sets defined by binomial inequalities. We also show
  that the wild behaviour of certain examples constructed by Krug
  \cite{kr} and by Mondal-Netzer \cite{mn} cannot occur in a toric
  setting.
\end{abstract}

\section*{Introduction}

The simplest measure for the asymptotic growth of a real polynomial
in $n$ variables on $\R^n$ is its total degree.
However, when we pass from $\R^n$ to an unbounded \sa\ subset $S\subset
\R^n$, the total degree of a polynomial may not reflect the growth
of the restriction $f|_S$ any more.

The degree of a polynomial can be understood
as its pole order along the hyperplane at infinity when $\R^n$
is embedded into projective space in the usual way. How
this relates to the growth of $f|_S$ depends on how the closure of $S$
in $\P^n(\R)$ meets the hyperplane at infinity. Unless this
intersection is empty (which would mean that $S$ is bounded in
$\R^n$) or of maximal dimension, the total degree alone will usually
not suffice to understand the growth of polynomials on $S$. We may
however hope to improve control of the growth by suitable blow-ups at
infinity.

To make this idea more precise, we consider the following
setup. Suppose that $V$ is an affine real variety and $S$ the
closure of an open semi-algebraic subset of $V(\R)$. An open-dense
embedding of $V$ into a complete variety $X$ is called
\emph{compatible with $S$} if the geometry of $S$ at infinity is
regular in the following sense: If $Z$ is any hypersurface at
infinity, i.e.~any irreducible component of the complement of $V$ in
$X$, the closure $\ol S$ of $S$ in $X(\R)$ meets $Z$ either in a
Zariski-dense subset of $Z$ or not at all. Under this condition, the
pole orders of a regular function $f$ on $V$ along the hypersurfaces
at infinity intersecting $\ol S$ accurately reflects the qualitative
growth of $f$ on~$S$.

Compatible completions were introduced by the authors in \cite{ps}, as
well as in the dissertation of the first author, motivated by earlier
work of Powers and Scheiderer in \cite{pws}. An $S$-compatible
completion of $V$ yields in particular a description of
\[
B_V(S)=\{f\in\R[V]\colon \exists \lambda\in\R\ |f|\le\lambda\text{ on S}\},
\]
the ring of regular functions on $V$ that are bounded on $S$.
If $V\into X$ is such an $S$-compatible completion and $Y$ is the
union of those irreducible components of $X\setminus V$ that are
disjoint from $\ol{S}$, then $B_V(S)$ is naturally identified with
$\scrO_X(X\setminus Y)$, the ring of regular functions of the variety
$X\setminus Y$.

The main goal of this paper is to improve on the results in \cite{ps}
and make them more explicit in the controlled setting of toric
varieties. Specifically, we study the following questions.

(1) One of the main results of \cite{ps} is the existence of regular
completions in the case $\dim V\le 2$. The higher-dimensional case
remains open and hinges on the existence of a certain type of embedded
resolution of singularities. In the toric setting, we introduce a
stronger, purely combinatorial compatibility condition in the spirit
of toric geometry (Section \ref{sec:Toric}). We show that this
condition can be satisfied if $S$
is defined by binomial inequalities (Corollary \ref{bdpolbinom}) or if $S$
is what we call a tentacle (Corollary \ref{bdpoltent}), generalizing a
concept introduced by Netzer in \cite{ne}.

Since the compatible completions in dimension $2$ constructed in
\cite{ps} are built from an embedded resolution of singularities, they
are typically quite hard to compute explicitly. By contrast, our
results in the toric setting only require the usual arithmetic of
semigroups derived from rational polyhedral cones.

(2) The transcendence degree of the ring of regular functions
$\scrO(X\setminus Y)$ of the complement of a divisor $Y$ in a complete
variety $X$ is called the \emph{Iitaka dimension} of $Y$. It is a
natural generalization of the Kodaira dimension studied extensively in
complex algebraic geometry. Thus in the case of a compatible
completion, when $B_V(S)$ is identified with $\scrO(X\setminus Y)$,
the Iitaka dimension measures in how many independent directions the
set $S$ is bounded.

In dimension $2$, the Iitaka dimension is strongly related to the
signature of the intersection matrix $A_Y$ of the divisor
$Y$. However, the correspondence is not perfect if $A_Y$ is singular.
Specifically, if $A_Y$ is negative semidefinite but not definite,
Iitaka's criterion (Proposition \ref{signintmat})
does not give anything. In the toric setting, on the other hand, we
show that the signature of $A_Y$ is sufficient to determine the
Iitaka dimension
(Proposition \ref{zsfsg}). It seems plausible that this has been
observed before,
but we were unable to find any trace in the literature. We
exploit the result in an application to positive polynomials
explained below.

(3) The existence of an $S$-compatible completion $X$ of $V$ yields a
good description of the ring of bounded functions $B_V(S)$. However,
it does not imply that $B_V(S)$ is a finitely generated
$\R$-algebra. This was discussed in \cite{ps} and much further
explored by Krug in \cite{kr}. When a toric $S$-compatible completion
exists, $B_V(S)$ is always finitely generated (Proposition
\ref{abstrfw}).

Beyond bounded polynomials, an $S$-compatible completion also provides
control over the asymptotic growth of arbitrary polynomials, as
indicated in the beginning. Let $Y'$ be the union of all irreducible
components of $X\setminus V$ that intersect the closure of $S$ in
$X$. In Section \ref{sec:Filtrations}, we study the linear subspaces
\[
  L_{X,m}(S)=\{f\in\R[V]\colon \text{all poles of }f\text{ along
    }Y'\text{ have order at most }m\},
\]
which consists of functions of bounded growth on $S$. Assume that
$B_V(S)=\R$. In analogy with the case $S=\R^n$, one might expect that
the spaces $L_{X,m}(S)$ ($m\in\N$) are finite-dimensional. If so, the
filtration $L_{X,0}(S)\subset L_{X,1}(S)\subset
L_{X,2}(S)\subset\cdots$ of $\R[V]$ behaves much like the filtration
of the polynomial ring by total degree. The properties of filtrations
obtained in this way and further generalizations have also been
studied in complex algebraic geometry (see Mondal
\cite{mo}). For us, this question is particularly relevant in the
context of positive polynomials and the moment problem, as it concerns
possible degree cancellations in sums of positive polynomials, as
explained in Section \ref{sec:PositivePolynomials}. However, a subtle
example due to Mondal and Netzer in \cite{mn} (see Example
\ref{Example:MondalNetzer}) implies that the $L_{X,m}(S)$ may have
infinite dimension. This construction is complemented by our Theorem
\ref{Thm:TotalStabilityToric}, which combines with the results of
\cite{sc} to show that if $S$ is basic open of dimension at least $2$
and admits an $S$-compatible toric completion but no non-constant
bounded function, then the spaces $L_{X,m}(S)$ are finite-dimensional
and, consequently, the moment problem for $S$ is not finitely
solvable. This comprises the results of Netzer in \cite{ne} for
tentacles and of Powers and Scheiderer in \cite{pws}. It is also
related to a theorem of Vinzant in \cite{vi}, which constructs a
certain kind of toric compatible completion under an algebraic
assumption on the description of $S$ and the ideal of $V$, as
explained in the last section of \cite{vi}.

\bigskip \emph{Acknowledgements.} We would like to thank Sebastian
Krug and Tim Netzer for helpful discussions, as well as the referee
for useful remarks on the manuscript. The first author
was supported by DFG grant PL\,549/3-1 and the Zukunftskolleg of the
University of Konstanz, the second author by DFG grant SCHE\,281/10-1.


\section{Compatible completions of semi-algebraic sets}\label{sec:CompatibleCompletions}

We briefly summarize some of the definitions and results in \cite{ps}.

\begin{dfn}\label{DefCompatibleCompletion}%
  Let $V$ be a normal affine $\R$-variety and let $S$ be a semi-algebraic
  subset of $V(\R)$. By a \emph{completion} of $V$ we mean an open
  dense embedding $V\into X$ into a normal complete
  $\R$-variety. The completion $X$ is said to be \emph{compatible
    with $S$} (or \emph{$S$-compatible}) if for every irreducible
  component $Z$ of $X\setminus V$ the following condition holds: The
  set $Z(\R)\cap\ol{S}$ is either empty or Zariski-dense in
  $Z$.
\end{dfn}

Here, when taking the closure $\ol{S}$ of a semi-algebraic subset $S$
of $X(\R)$, we refer to the Euclidean topology on $X(\R)$, rather
than the Zariski topology. Note that every irreducible component of
$X\setminus V$ is a divisor on $X$, i.e., has codimension one
(\cite{ha156} p.~66).

\begin{thm}[{\cite[Thm.~3.8]{ps}}]\label{CompatibleCompletionIsoBounded}%
Let $V$ be a normal affine $\R$-variety, let $S\subset V(\R)$ be a
semi-algebraic subset, and assume that the completion $V\into X$ of $V$ is
compatible with $S$. Let $Y$ denote the union of those irreducible
components $Z$ of $X\setminus V$ for which $\ol S\cap Z(\R)=
\emptyset$, and put $U=X\setminus Y$. Then the inclusion $V\subset U$
induces an isomorphism of $\R$-algebras
\[
\scrO_X(U)\>\cong B_V(S). \eqno{\qed}
\]
\end{thm}

A semi-algebraic set is called \emph{regular} if its closure coincides
with the closure of its interior. It is called \emph{regular at
  infinity} if it is the union of a regular and a relatively compact
semi-algebraic set. One of the main results of \cite{ps} is the
existence of compatible completions for two-dimensional semi-algebraic
sets regular at infinity.

\begin{thm}\cite[Thm.~4.5]{ps} \label{ExistenceCompletionSurfaces}
  Let $V$ be a normal quasi-projective surface over $\R$,
  and let $S$ be a semi-algebraic subset of $V(\R)$ that is regular at
  infinity. Then $V$ has an $S$-compatible projective completion. If
  $V$ is non-singular then the completion can be chosen to be
  non-singular as well.\qed
\end{thm}

\begin{lab}\label{constrgoodcompl}%
The proof of Theorem \ref{ExistenceCompletionSurfaces} is essentially
constructive and relies on embedded resolution of singularities. We
summarise the procedure for our present purposes. Let $V\into X$ be
any open-dense embedding of $V$ into a normal projective surface. Let
$C_\partial$ be the Zariski-closure of the boundary of $S$ in $X(\R)$
and let $C_\infty=X\setminus V$. Put $C=C_\partial\cup C_\infty$, a
reduced curve in $X$. We write
\[
\partial_X^\infty S=\ol{S}\cap C_\infty(\R)
\]
for the set of \emph{boundary points of $S$ at infinity} in $X$. A sufficient
condition for $X$ to be an $S$-compatible completion of $V$ is that
$C$ has \emph{only normal crossings} in $\partial_X^\infty
S$. Explicitly, this means the following. If $P\in\partial_X^\infty
S$, then
\begin{itemize}
\item[(1)]
$P$ is a non-singular point of all irreducible components of $C$
  that contain it.
\item[(2)]
$P$ is contained in exactly one component $C_0$ of $C_\partial$
  and one component $C_1$ of $C_\infty$, and $C_0(\R)$ and $C_1(\R)$
  have independent tangents in $P$. (Equivalently, the local equations
  for $C_0$ and $C_1$ generate the maximal ideal of the local ring
  $\scrO_{X,P}$.)
\end{itemize}

By blowing up any points in $\partial_X^\infty S$ in which conditions
(1) or (2) are violated and proceeding inductively, we can produce a
completion $\wt X$ and a corresponding curve $\wt C=\wt
C_\partial\cup\wt C_\infty$, defined as before, such that all points
of $\partial_{\wt X}^\infty S$ are normal crossings of $\wt C$, which
is therefore an $S$-compatible completion of $V$. Note that blowing-up
increases the number of irreducible components in $C_\infty$, since
the exceptional divisor is added. In the resulting $S$-compatible
completion $X$, the divisor $Y$ of Theorem~\ref{CompatibleCompletionIsoBounded}
consists of those
irreducible components of $C_\infty$ which are disjoint from
$\ol{S}$.
\end{lab}

Explicit computation of the ring of bounded polynomials following the
above procedure is possible but can quickly turn into a cumbersome
task. We give the following simple example as illustration.
A much more interesting example will be discussed in
Section~\ref{sec:Filtrations}.

\begin{example}\label{example:CompatibleCompletionStrip}
  Let $S=\{(x,y)\in\R^2\colon -1\le x\le 1\}$ be a strip in the affine
  plane $V=\A^2_\R$ and consider the embedding $V\into\P^2_\R$ into
  the projective plane given by $(u,v)\mapsto (u:v:1)$. Then
  $C_\infty=\P^2\setminus V$ is the line at infinity and $C_\partial$
  is the Zariski-closure of the two lines $\scrV(x-1)$ and $\scrV(x+1)$ in
  $V$. The set $\partial S^\infty_{\P^2}$ is the point $P=(0:1:0)$,
  which is also the unique intersection point of $C_\partial$ and
  $C_\infty$. In local coordinates $r=x/y$ and $s=1/y$ of $\P^2_\R$
  centered around $P$, we have $C_\infty=\scrV(s)$ and
  $C_\partial=\scrV\bigl((r-s)(r+s)\bigr)$. Since all three components
  of $C=C_\infty\cup C_\partial$ pass through $P$, $C$ does not have
  normal crossings in $P$. Indeed, the completion of $S$ is not
  $S$-compatible, since $\ol{S}\cap C_\infty(\R)=\{P\}$ is not
  Zariski-dense in $C_\infty$.

  Let $\wt X$ be the blow-up of $\P^2_\R$ in $P$. It is given in local
  coordinates by the quadratic transformation $r=r_1$, $s=r_1s_1$. In
  the new coordinates $r_1,s_1$, the exceptional divisor is
  $E=\scrV(r_1)$. The strict transforms of the components of $C$ in
  $X$ are $C_\infty=\scrV(s_1)$ and
  $C_\partial=\scrV\bigl((s_1-1)(s_1+1)\bigr)$. Now $\wt C_\infty=\wt
  X\setminus V$ has the two components $C_\infty$ and $E$. Since
  $C_\partial$ meets $E$ in the points $(0,1)$ and $(0,-1)$ but does
  not meet $C_\infty$, we see that $\wt X$ is an $S$-compatible
  completion of $V$ and $Y=C_\infty$ is the component of $\wt
  C_\infty$ that is disjoint from $\ol{S}_{\wt X(\R)}$. To compute
  $\scrO(\wt X\setminus Y)$, write $f\in\R[x,y]$ as $f=\sum_{i,j}
  a_{ij}x^iy^j= \sum_{i,j} a_{ij}r_1^{-j}s_1^{-i-j}$, so that $f$ lies
  in $\scrO(\wt X\setminus Y)$ if and only if $j=0$. Thus
  $B(S)=\scrO(\wt X\setminus Y)=\R[x]$.
\end{example}

In dimensions $\ge3$, it is not even guaranteed that the ring
$B(S)$ is finitely generated (see \cite{ps} Sect.~5).


\section{Toric completions}\label{sec:Toric}

Let $V$ be an affine toric variety. By a \emph{toric completion} of
$V$ we mean an open embedding of $V$ into a complete toric variety
$X$ which is compatible with the torus actions. Let $S\subset V(\R)$
be a semi-algebraic subset. We are going to work out conditions on $S$
ensuring that $V$ has a toric completion $V\subset X$ that is
compatible with $S$. The existence of such a completion allows us to
make the ring of bounded polynomial functions on $S$ completely
explicit. It also prevents several pathologies that can occur in more
general cases.

We start by reviewing some general notions on toric varieties. An
excellent reference is the book of Cox, Little and Schenck \cite{cls}.

\begin{lab}
Let $T$ be an $n$-dimensional split $\R$-torus and let $T(\R)\cong
(\R^*)^n$ be the group of $\R$-points. All toric varieties will be
$T$-varieties. Let $M=\Hom(T,\G_m)$ (resp.\ $N=\Hom(\G_m,T)$), the
group of characters (resp.\ of co-characters) of $T$. Both are free
abelian groups of rank~$n$, each being the natural dual of the other.
We write both groups additively and denote the character corresponding
to $\alpha\in M$ by $\x^\alpha$, the co-character corresponding to
$v\in N$ by $\lambda_v$.
The pairing between $M$ and $N$ will be denoted $\bil\alpha v$.
\end{lab}

\begin{lab}\label{rappl2}%
Let $M_\R=M\otimes\R$, $N_\R=N\otimes\R$. By a cone $\sigma\subset
N_\R$ we always mean a finitely generated rational convex
cone. Let $\sigma^*\subset M_\R$ denote the dual cone of
$\sigma$, let $H_\sigma=M\cap\sigma^*$, and write $\R[H_\sigma]$ for
the semigroup algebra of $H_\sigma$. Then
$U_\sigma=\Spec\R[H_\sigma]$ is an affine toric variety that
contains a unique closed $T$-orbit, denoted $O_\sigma$.

Assume that the cone $\sigma\subset N_\R$ is pointed. Then the dense
$T$-orbit $U_0$ in $U_\sigma$ is isomorphic to $T$, and we may use any
fixed $\xi_0\in U_0(\R)$ to equivariantly identify $U_0$ with $T$. Let
$v\in N\cap\relint(\sigma)$. For any $\xi\in U_0(\R)$, the limit
\[
L_v(\xi)\>:=\>\lim_{s\to0}\Bigl(\lambda_v(s)\cdot\xi\Bigr)
\]
exists in $U_\sigma(\R)$ and lies in $O_\sigma(\R)$.
Clearly, the map $L_v\colon U_0(\R)\to O_\sigma(\R)$ is
equivariant under the $T(\R)$-action. In particular, $L_v$ is an open
map.
\end{lab}

\begin{lab}
Fixing $v\in N$, we consider the $v$-grading of $\R[T]$, which is
the grading that makes the character $\x^\alpha$ homogeneous of
degree $\bil\alpha v$, for every $\alpha\in M$. We say that
$f\in\R[T]$ is $v$-homogeneous if $f$ is homogeneous in the
$v$-grading. For $0\ne f\in\R[T]$, let $\lf_v(f)\in\R[T]$ denote the
leading component of $f$ in the $v$-grading, i.e., the non-zero
$v$-homogeneous component of $f$ of \emph{smallest} $v$-degree.
Two vectors $v,\,v'\in N$ satisfy $\lf_v(f)=\lf_{v'}(f)$ if and
only if $v$ and $v'$ lie in the relative interior of the same cone of
the normal fan of the Newton polytope of $f$. (Note that since we
define the leading form $\lf_v(f)$ to be the homogeneous component of
smallest $v$-degree, we are using inward, rather than outward, normal
cones here.)
\end{lab}

\begin{lab}
A fan is a finite non-empty set $\Sigma$ of closed pointed rational
cones in $N_\R$ which is closed under taking faces and such that the
intersection of any two elements of $\Sigma$ is a face of both. The
union of all cones in $\Sigma$ is called the support of $\Sigma$,
denoted by $|\Sigma|$; if $|\Sigma|=N_\R$, then $\Sigma$ is called
complete.
The fan $\Sigma$ gives rise to a toric variety $X_\Sigma$, obtained
by glueing the affine toric varieties $U_\sigma$, $\sigma\in\Sigma$.
The variety $X_\Sigma$ is complete if and only if the fan $\Sigma$ is
complete. In general, the ring of global regular functions
$\scrO(X_\Sigma)$ is the semigroup algebra $\R[H]$, where
$H=M\cap |\Sigma|^\ast$.
By Dickson's lemma, this is a finitely generated $\R$-algebra.
\end{lab}

Let $U_\sigma$ be an affine toric variety, and let $S\subset
U_\sigma(\R)$ be a \sa\ set. We are going to study conditions under
which there exists a toric completion of $U_\sigma$ that is
compatible with $S$, and which therefore
allows the explicit computation of the ring $B_{U_\sigma}(S)$ of
polynomials bounded on $S$. We first propose an abstract framework,
see Proposition \ref{abstrfw} below. After this, we will exhibit
concrete situations to which the abstract framework applies.

We will always assume that the \sa\ set $S$ is open
and contained in the dense torus orbit in $U_\sigma$.

\begin{lab}\label{dfnwv}%
Let $S\subset T(\R)$ be an open \sa\ subset. Given $v\in N$, put
$$S(v)\>:=\>\bigl\{\xi\in T(\R)\colon\all0<s\ll1\ \ \lambda_v(s)\xi
\in S\bigr\}.$$
It is easily seen that $(S_1\cup S_2)(v)=S_1(v)\cup S_2(v)$ and
$(S_1\cap S_2)(v)=S_1(v)\cap S_2(v)$ hold for all $v\in N$ and all
open \sa\ sets $S_1,\,S_2\subset T(\R)$. Further let
$$K(S)\>:=\>\{v\in N\colon S(v)\ne\emptyset\},\quad K_0(S)\>:=\>
\{v\in N\colon\inter(S(v))\ne\emptyset\}.$$
Then $K(S_1\cup S_2)=K(S_1)\cup K(S_2)$ and $K_0(S_1\cup S_2)=
K_0(S_1)\cup K_0(S_2)$ hold.
\end{lab}

\begin{lem}\label{adaptfan}%
  Given any open \sa\ set $S\subset T(\R)$, there exists a
  fan $\Sigma$ in $N_\R$ such that
\[
K(S)\>=\>N\cap\bigcup_{\sigma\in E}\relint(\sigma),\quad K_0(S)
\>=\>N\cap\bigcup_{\sigma\in E_0}\relint(\sigma)
\]
hold for suitable subsets $E_0$, $E$ of $\Sigma$. Any such fan
$\Sigma$ is said to be \emph{adapted to~$S$}.
\end{lem}

\begin{proof}
We may assume that $S=\{\xi\in T(\R)\colon f_i(\xi)>0$
($i=1,\dots,r)\}$ is basic open, with $f_1,\dots,f_r\in\R[T]$.
Given $f\in\R[T]$ and $v\in N$, let $f_{v,d}\in\R[T]$ be the
$v$-homogeneous component of $f$ of degree~$d$. Thus
$$f(\lambda_v(s)\xi)\>=\>\sum_{d\in\Z}f_{v,d}(\xi)\cdot s^d$$
for $s\in\R$ and $\xi\in T(\R)$.
So $\xi\in S(v)$ holds if and only if, for every $i=1,\dots,r$, there
exists $d_i\in\Z$ with $(f_i)_{v,d_i}(\xi)>0$ and with $(f_i)_{v,d'}
(\xi)=0$ for all $d'<d_i$.
Let $\Lambda(f,v)$ denote the sequence of non-zero $v$-homogeneous
components of $f$, ordered by increasing degree.
Then, if $v,\,v'\in N$ satisfy $\Lambda(f_i,v)=\Lambda(f_i,v')$
for $i=1,\dots,r$, if follows that $S(v)=S(v')$. It is clear that
there is a fan $\Sigma$ such that any two vectors $v,\,v'$ in the
relative interior of the same cone of $\Sigma$ satisfy this
condition. Such $\Sigma$ satisfies the condition of the lemma.
\end{proof}

\begin{rem}
  If $S$ is a subset of the positive orthant in $\R^n$, the
  set $K(S)$ is closely related to the tropicalization of $S$
  constructed by Alessandrini in \cite{al}.
\end{rem}

\begin{lem}\label{klarlem}%
Let $S\subset T(\R)$ be an open \sa\ set, and let $\rho\subset N_\R$
be a pointed cone.
\begin{itemize}
\item[(a)]
If $K(S)\cap\relint(\rho)\ne\emptyset$, then $\ol S\cap O_\rho(\R)
\ne\emptyset$.
\item[(b)]
If $K_0(S)\cap\relint(\rho)\ne\emptyset$, then $\ol S\cap O_\rho
(\R)$ is Zariski dense in $O_\rho$.
\end{itemize}
\end{lem}

Here we fix an equivariant identification $T=U_0$. The closures are
taken inside the affine toric variety $U_\rho$ and with respect to
the Euclidean topology. Recall that $U_\rho$ contains $U_0=T$ (resp.\
$O_\rho$) as an open dense (resp.\ as a closed) $T$-orbit.

\begin{proof}
Given $v\in N\cap\relint(\rho)$, the map $L_v\colon T(\R)\to
O_\rho(\R)$ (see \ref{rappl2}) is open and maps $S(v)$ into
$\ol S\cap O_\rho(\R)$. The hypothesis $v\in K(S)$ resp.\ $v\in
K_0(S)$ means $S(v)\ne\emptyset$ resp.\ $\inter(S(v))\ne\emptyset$.
This proves the lemma.
\end{proof}

\begin{lab}\label{abstractsetup}%
Now let $\Sigma$ be a complete fan in $N_\R$, and let $X_\Sigma$ be
the associated complete toric variety. We write $\Sigma(d)$ for the
set of $d$-dimensional cones in $\Sigma$. For $\tau\in\Sigma$, let
$Y_\tau$ be the Zariski closure of $O_\tau$ in $X_\Sigma$.
In particular, $Y_\tau$ is a prime Weil divisor on $X_\Sigma$ when
$\tau\in\Sigma(1)$. We fix a cone $\sigma\in\Sigma$ and consider
$X_\Sigma$ as a toric completion of the affine toric variety
$U_\sigma$.

Let $S\subset T(\R)=U_0(\R)$ be an open \sa\ set. We will require the
following \emph{toric compatibility} assumption:
\begin{center}
  \begin{tabular}{lp{0.7\textwidth}}
  \emph{(TC)} & \emph{For any $\tau\in\Sigma(1)$ with
  $\tau\not\subset\sigma$, either $\ol S\cap Y_\tau(\R)$ is empty or
  $K_0(S)\cap\relint(\tau)$ is non-empty.}
  \end{tabular}
\end{center}
(The two cases are mutually exclusive by Lemma \ref{klarlem}.)
We define the subfan $F_S$ of $\Sigma$ by
\[
F_S \>:=\> \left\{\rho\in\Sigma\colon
  \text{\begin{tabular}{l}
Every one-dimensional face $\tau$ of $\rho$ satisfies\\
$\tau\subset\sigma$ or $K_0(S)\cap\relint(\tau)\ne\emptyset$
\end{tabular}}
\right\}.
\]
\end{lab}

\begin{prop}\label{abstrfw}%
With the above notation, assume that the toric compatibility condition
$(TC)$ holds. Then the toric variety $X_\Sigma$ is an
$S$-compatible completion of $U_\sigma$. In particular, let
$B_{U_\sigma}(S)$ be the subring of $\R[U_\sigma]$ consisting of the
regular functions that are bounded on $S$. Then
$$B_{U_\sigma}(S)\>=\>\scrO(X_{F_S})\>=\>\R[H]$$
with $H=M\cap|F_S|^*$. In particular, the $\R$-algebra
$B_{U_\sigma}(S)$ is finitely generated.
\end{prop}

\begin{proof}
Write $X=X_\Sigma$, a normal and complete toric variety containing
$U_\sigma$ as an open affine toric subvariety. The irreducible
components of $X\setminus U_\sigma$ are the $Y_\tau$ where $\tau\in
\Sigma(1)$ and $\tau\not\subset\sigma$.
Given such $\tau$ with $\ol S\cap Y_\tau(\R)\ne\emptyset$, we know
that $\ol S\cap Y_\tau(\R)$ is Zariski dense in $Y_\tau$, by
condition $(TC)$ and Lemma \ref{klarlem}(b).
So the completion $X$ of $U_\sigma$ is compatible with the \sa\ set
$S\subset U_\sigma(\R)$, in the sense of
\ref{DefCompatibleCompletion}.
By Theorem \ref{CompatibleCompletionIsoBounded},
we therefore have $B_{U_\sigma}(S)=
\scrO(X\setminus Y)$ where $Y$ is the union of those irreducible
components $Y_\tau$ of $X\setminus U_\sigma$ for which $\ol S\cap
Y_\tau(\R)=\emptyset$. By condition $(TC)$, the latter means
$\tau\in\Sigma(1)$, $\tau\not\subset\sigma$ and $\relint(\tau)\cap
K_0(S)=\emptyset$. So $X\setminus Y$ is precisely the toric variety
associated to the subfan $F_S$ of $\Sigma$ defined above.
\end{proof}

\begin{examples}\label{examples:CompatibleCompletions}
Let $n=2$. We compatibly identify $M=\Z^2$, $N=\Z^2$ and
$T(\R)= (\R^\ast)^2$. We denote by $(e_1,e_2)$ the standard basis of
$N_\R=\R^2$ and by $(e_1^\ast,e_2^\ast)$ the dual basis of
$M_\R$. Let
$\sigma=\cone(e_1,e_2)$ be the positive quadrant in $N_\R$, so that
$\R[U_\sigma]=\R[x_1,x_2]$ and $U_\sigma=\A^2$.
Let $\Sigma_0$ be the standard fan of $\P^2$ with ray generators
$e_1$, $e_2$ and $-(e_1+e_2)$.
We use homogeneous coordinates $(u_0:u_1:u_2)$ on $\P^2$ with $x_i=
\frac{u_i}{u_0}$ ($i=1,2$).
\smallskip

(1)
Consider the set
$$S\>:=\>\bigl\{(\xi_1,\xi_2)\in(\R^*)^2\colon -1<\xi_1< 1\bigr\}.$$
It is easily seen that $K(S)=K_0(S)=\{(v_1,v_2)\in N\colon
v_1\ge0\}$.
Let $\Sigma$ be the refinement of $\Sigma_0$ generated by the
additional ray generator $-e_2$. Then $\Sigma$ is adapted to $S$,
c.f.\ Lemma \ref{adaptfan}. The toric variety $X_\Sigma$ is the
blowup of $\P^2$ in the point $(0:0:1)$, which is exactly the
compatible completion of the strip $\ol S\subset\R^2$ we
considered in Example \ref{example:CompatibleCompletionStrip}.
By definition, $|F_S|=\{(v_1,v_2)\in N_\R\colon v_1\ge0\}$,
so that $M\cap|F_S|^\ast=\{(k,0)\in M\colon k\ge0\}$, whence
$\scrO(X_{F_S})=\R[x_1]$. It is not hard to check that condition
$(TC)$ is met in this example, so that Proposition \ref{abstrfw}
yields $B_{\A^2}(S)=\R[x_1]$.
\smallskip

(2)
Let $k\ge1$, and let
$$S\>:=\>\bigl\{(\xi_1,\xi_2)\in \R^2\colon\xi_1^k\xi_2<1,\
\xi_1>0,\ \xi_2>0\bigr\}$$
(see Figure \ref{fig:FanHyperbola1} for $k=2$). Here we find that
$K(S)=K_0(S)$ is the half-space $kv_1+v_2\ge0$ in $N$.
We again refine $\Sigma_0$ by adding ray generators $\pm(e_1-ke_2)$
to $\Sigma_0$, and obtain the fan $\Sigma$ shown on the right of
Figure \ref{fig:FanHyperbola1}.

By construction, $\Sigma$ is adapted to $S$, and $|F_S|=
\{v\in N_\R\colon kv_1+v_2\ge0\}$.
We check that condition $(TC)$ is satisfied. This amounts to
showing for $\tau=\cone(-e_1-e_2)$ that $Y_\tau(\R)\cap\ol{S}=
\emptyset$.
Indeed, let $\rho=\cone(-e_1-e_2,\,-e_1+ke_2)$. Then
$\rho^*$ is generated by $e_2^*-e_1^*$ and $-(ke_1^*+e_2^*)$, so
that $\R[U_\rho]=\R[H]$, where $H$ is the saturated semigroup
generated by $y_1=x_1^{-k}x_2^{-1}$, $y_2=x_1^{-1}$ and $y_3=
x_1^{-1}x_2$, so that $\R[U_\rho]\cong\R[y_1,y_2,y_3]/
(y_1y_3-y_2^{k+1})$.
Under this identification, we find $y_1=0$ on $Y_\tau\cap U_\rho$
while $y_1>1$ on $S\cap U_\rho(\R)$.
So $\ol{S}\cap(U_\rho\cap Y_\tau)(\R)=\emptyset$. Essentially the
same computation applies to $\rho'=\cone(-e_1-e_2,\,e_1-ke_2)$. Hence
we conclude $B_{\A^2}(S)=\scrO(X_{F_S})=\R[x_1^kx_2]$.
This
will be discussed in general below (c.f.~Corollary~\ref{bdpolbinom}).

\begin{figure}[h]
  \centering
  \begin{tikzpicture}[scale=0.5]
    \begin{scope}
      \clip (-0.5,-0.5) rectangle (5.5,4.5);
      \draw[smooth,domain=0.1:4.5,thick, samples=\nos] plot({\x},{1/\x^2))});
    \pgfsetfillpattern{north west lines}{black};
       \fill[smooth,domain=0.1:4.5,thick, samples=\nos] (0,0) --
       (0,4.5) --
       plot({\x},{1/\x^2))}) -- (4.5,0) -- (0,0);
    \end{scope}
    \draw[->] (-0.5,0) -- (4.5,0) node[right]{$x_1$}; \draw[->] (0,-0.5) --
    (0,4.5) node[above]{$x_2$};
  \end{tikzpicture}\hspace{5em}
  \begin{tikzpicture}[scale=1.2]
    \draw[thick] (0,0) -- (1,0); \draw[thick] (0,0) -- (0,1);
    \pgfsetfillpattern{north west lines}{black};
    \fill (0,0) -- (1,0) -- (1,1) -- (0,1) -- (0,0);

    \draw[thick] (0,0) -- (-1/2,1);
    \pgfsetfillpattern{north east lines}{black};
    \fill (0,0) -- (-1/2,1) -- (0,1) -- (0,0);
    \fill (0,0) -- (1,0) -- (1/2,-1) -- (0,0);

    \draw[thick] (0,0) -- (1/2,-1);

    \draw (0,0) -- (-1,-1);
  \end{tikzpicture}\caption{}\label{fig:FanHyperbola1}
\end{figure}
\smallskip

(3)
Let
$$S\>:=\>\bigl\{(\xi_1,\xi_2)\in \R^2\colon\xi_1(\xi_1-\xi_2)+1>0,\
\xi_2(\xi_2-\xi_1)+1>0,\ \xi_1>0,\ \xi_2>0\bigr\}$$
(see Figure \ref{fig:FanHyperbola2}). In this example, we have
$K_0(S)=\{v\in N\colon v_1+v_2\ge0\}$, while $K(S)$ consists of
$K_0(S)$ and the half-line $\tau$ generated by $-(e_1+e_2)$.
Let $\Sigma$ be the complete fan with ray generators $e_1$, $e_2$,
$\pm(e_1-e_2)$ and $-(e_1+e_2)$ (see Figure \ref{fig:FanHyperbola2}).
Then $\Sigma$ is adapted to $S$ and $|F_S|$ is the half-plane
$v_1+v_2\ge0$. But condition ($TC$) is not satisfied. If it were,
we could conclude $B_{\A^2}(S)=\R[x_1x_2]$, which is clearly not true
since $x_1x_2$ is unbounded on $S$.
Indeed, it follows from Lemma \ref{klarlem} that
$Y_\tau(\R)\cap\ol{S}\neq\emptyset$. (It is not hard to check
directly that it is a single point).
\begin{figure}[h]
  \centering
  \begin{tikzpicture}[scale=0.5]
    \begin{scope}
      \clip (-0.5,-0.5) rectangle (4.5,4.5);

    \pgfsetfillpattern{north west lines}{black};
      \fill (0,0) rectangle (4.5,4.5);

       \fill[white, smooth,domain=0.1:4.5,thick, samples=\nos]
       plot({\x},{\x+1/\x});
      \fill[white, smooth,domain=0.1:4.5,thick, samples=\nos]
       plot({\x+1/\x},{\x});

      \draw[smooth,domain=0.1:4.5,thick, samples=\nos]
      plot({\x},{\x+1/\x});
      \draw[smooth,domain=0.1:4.5,thick, samples=\nos]
      plot({\x+1/\x},{\x});
    \end{scope}
    \draw[->] (-0.5,0) -- (4.5,0) node[right]{$x_1$}; \draw[->] (0,-0.5) --
    (0,4.5) node[above]{$x_2$};
  \end{tikzpicture}\hspace{5em}
  \begin{tikzpicture}[scale=1.2]
    \draw[thick] (0,0) -- (1,0); \draw[thick] (0,0) -- (0,1);
    \pgfsetfillpattern{north west lines}{black};
    \fill (0,0) -- (1,0) -- (1,1) -- (0,1) -- (0,0);

    \draw[thick] (0,0) -- (-1,1);
    \draw[thick] (0,0) -- (1,-1);
    \pgfsetfillpattern{north east lines}{black};
    \fill (0,0) -- (-1,1) -- (0,1) -- (0,0);
    \fill (0,0) -- (1,0) -- (1,-1) -- (0,0);

    \draw (0,0) -- (-1,-1);
  \end{tikzpicture}\caption{}\label{fig:FanHyperbola2}
\end{figure}
\smallskip

(4)
The key property for Proposition \ref{abstrfw} to apply is
condition $(TC)$ from \ref{abstractsetup}. For a general open \sa\
set, this condition cannot be satisfied by any choice of a fan
$\Sigma$ in $N_\R$, as is demonstrated by the following simple
example: For the open set $S=\{\xi\in(\R_\plus^*)^2\colon\xi_1,\,
\xi_2>1$ and $1<\xi_1-\xi_2<2\}$ in the 2-dimensional torus, we have
$K(S)=\Z_\plus(-e_1-e_2)$ and $K_0(S)=\{0\}$. So, at least with
respect to $\sigma=\{0\}$, condition $(TC)$ cannot hold for any
complete fan $\Sigma$.
\end{examples}

\begin{rem}
  Example (4) can still be saved by making a linear change
  of coordinates.
  However, it is clear that more complicated examples
  of open sets may be constructed for which no linear coordinate
  change allows to apply condition $(TC)$.
  There is also an indirect way to see this. Whenever
  condition $(TC)$ applies, we see from Proposition \ref{abstrfw} that
  the ring $B_{U_\sigma}(S)$ of bounded polynomials on $U_\sigma$ is
  finitely generated as an $\R$-algebra.
  On the other hand, it is
  known that there exist open \sa\ subsets $S$ of $(\R^*)^n$ for
  $n\ge3$ for which the $\R$-algebra $B_{\A^n} (S)$ fails to be
  finitely generated (see Krug \cite{kr}).
\end{rem}

Condition ($TC$) can be rather cumbersome to check, as the
above examples show. We therefore seek favorable situations in which
this condition can be guaranteed, and therefore allows a purely
combinatorial computation of the ring of bounded functions. We
discuss two classes of sets where this approach is successful, namely
binomially defined sets and the so-called ``tentacles'' considered by
Netzer in \cite{ne}.

\begin{lab}
Let $Q:=(\R^*_\plus)^n\subset T(\R)$, and let
$$S\>=\>\Bigl\{\xi\in Q\colon a_i\xi^{\alpha_i}<b_i\xi^{\beta_i}\
(i=1,\dots,r)\Bigr\}$$
be a non-empty basic open set in $Q$ defined by binomial
inequalities, where $0\ne a_i,\,b_i\in\R$ und $\alpha_i,\,\beta_i\in
M=\Z^n$ ($i=1,\dots,r$).
An easy argument shows that the inequalities can be rewritten with
$a_i=1$ and $\beta_i=0$ for all~$i$.
For the following discussion we will therefore assume
$$S\>=\Bigl\{\xi\in Q\colon\xi^{\gamma_i}<c_i\ (i=1,\dots,r)\Bigr\}$$
where $\gamma_i\in M$ and $c_i>0$ ($i=1,\dots,r$).

We use the notation introduced in \ref{dfnwv}. Let $v\in N$. If
$\bil{\gamma_i}v>0$ for all $i$ then $S(v)=Q$. If $\bil{\gamma_i}v
\ge0$ for all $i$ then $S\subset S(v)$. If $\bil{\gamma_i}v<0$ for
some~$i$ then $S(v)=\emptyset$.
So we see that
$$K(S)\>=\>K_0(S)\>=\>C_S^*\cap N$$
where $C_S:=\cone(\gamma_1,\dots,\gamma_r)\subset M_\R$ and $C_S^*
\subset N_\R$ is the dual cone of $C_S$.

The next lemma contains the reason why condition $(TC)$ can be met:
\end{lab}

\begin{lem}\label{binomkeyobs}%
Let $\rho\subset N_\R$ be a pointed cone satisfying $\ol S\cap
O_\rho(\R)\ne\emptyset$. Then $C_S^*\cap\relint(\rho)\ne
\emptyset$.
\end{lem}

\begin{proof}
We may work in the toric affine variety $U_\rho=\Spec\R[H_\rho]$ with
$H_\rho=M\cap\rho^*$. Any point $\xi\in O_\rho(\R)$ satisfies
$\xi^\gamma=0$ for all $\gamma\in H_\rho\setminus(-H_\rho)$.
Let us write $\tau:=-C_S^*$, so that $H_\tau:=M\cap\tau^*$ is the
saturation inside $M$ of the semigroup generated by $-\gamma_1,\dots,
-\gamma_r$. Any $\beta\in H_\tau$ can be written in the form $\beta=
-\sum_{i=1}^rb_i\gamma_i$ with rational numbers $b_i\ge0$. Therefore,
there exists $c>0$ with $\xi^\beta>c$ for all $\xi\in S$.
Hence we have $\xi^\beta\ge c>0$ for any $\xi\in\ol S$, which implies
$\beta\notin H_\rho\setminus(-H_\rho)$. Thus $H_\tau\cap H_\rho
\subset-H_\rho$, or equivalently, by dualizing, $-\rho\subset
\rho+\tau$.
Choose any $u\in\relint(\rho)$. There exists $v\in\rho$ with $-u\in
v+\tau$, i.e.\ with $u+v\in-\tau=C_S^*$. This proves the lemma since
$u+v\in\relint(\rho)$.
\end{proof}

\begin{cor}\label{binomcond}%
Let $\Sigma$ be a complete fan in $N_\R$ which is adapted to $S$.
Then condition $(TC)$ from \ref{abstractsetup} is satisfied.
\end{cor}

\begin{proof}
Adapted simply means here that $C_S^*$ is a union of cones from
$\Sigma$. The claim is clear from Lemma \ref{binomkeyobs}: If
$\tau\in\Sigma(1)$ satisfies $\ol S\cap Y_\tau(\R)\ne\emptyset$, then
$\ol S\cap O_\rho(\R)\ne\emptyset$ for some $\rho\in\Sigma$ containing
$\tau$. By Lemma \ref{binomkeyobs}, this implies $C_S^*\cap
\relint(\rho)\ne\emptyset$. By adaptedness, this implies $\tau\subset
C_S^*$.
\end{proof}

We conclude that an $S$-compatible toric completion exists whenever
$S$ is defined by binomial inequalities:

\begin{cor}\label{bdpolbinom}%
Let $\sigma$ be a pointed cone in $N_\R$, and let $S=\{\xi\in Q\colon
\xi^{\gamma_i}<c_i$ $(i=1,\dots,r)\}$ as before, considered as a
subset of $U_\sigma(\R)$. The ring of polynomials on $U_\sigma$ that
are bounded on $S$ is given by
$$B_{U_\sigma}(S)\>=\>\R[H]$$
where $H=M\cap\sigma^*\cap C_S$.
\qed
\end{cor}

\begin{lab}
A polynomial function $f\in\R[U_\sigma]$ is therefore bounded on $S$
if and only if for every monomial $m$ occuring in $f$, some power of
$m$ is a product of $\x^{\gamma_1},\,\dots,\,\x^{\gamma_r}$. It is
obvious that such $f$ is bounded on $S$; the content of Corollary
\ref{bdpolbinom} is that no other $f$ is bounded on $S$. In
particular, we see that $B_{U_\sigma}(S)=\R$ if and only if $\sigma+
C_S^*=N_\R$.
\end{lab}

\begin{lab}
For a second class of examples, let $U$ be a non-empty open \sa\
subset of $Q=\{\xi\in T(\R)=(\R^*)^n\colon\xi_i>0$
$(i=1,\dots,n)\}$, and let $v\in N$. We consider the open set
$$S\>:=\>S_v(U)\>:=\>\bigl\{\lambda_v(s)\xi\colon\xi\in U,\
0<s\le1\bigr\}$$
in $Q$, which we may call a $v$-tentacle, following
Netzer in \cite{ne}.  Multiplying $v$ by a positive integer does not
change $S$, therefore we may assume that $v$ is a primitive element of
$N$.
\end{lab}

\begin{lem}
Assume that $U$ is relatively compact in $Q$. Let $S=S_v(U)$ be the
associated $v$-tentacle as above.
\begin{itemize}
\item[(a)]
$K(S)=K_0(S)=\Z_\plus v$.
\item[(b)]
If $\{0\}\ne\rho\subset N_\R$ is a pointed cone with $\ol S\cap
O_\rho(\R)\ne\emptyset$, then $v\in\relint(\rho)$.
\end{itemize}
\end{lem}

\begin{proof}
(a)
We obviously have $U\subset S(v)$, and therefore $v\in K_0(S)$.
Conversely let $u\in N$ with $S(u)\ne\emptyset$. So there is $\xi\in
Q$ such that $\lambda_u(s)\xi\in S$ for all sufficiently small real
$s>0$. Thus, for any small $s>0$ there exist $0<t\le1$ and $\eta\in
U$ with $\lambda_u(s)\xi=\lambda_v(t)\eta$. Assume $u\notin\R_\plus
v$. Then there exists $\alpha\in M$ with $\bil\alpha u>0>\bil\alpha
v$. Evaluating the character $\x^\alpha$ we get $s^{\bil\alpha u}
\xi^\alpha=t^{\bil\alpha v}\eta^\alpha\ge\eta^\alpha$. The right hand
side is positive and bounded away from zero, since $\x^\alpha$
does not approach zero on $U$. On the other hand, the left hand side
tends to zero for $s\to0$. This contradiction proves the claim.

(b)
The proof is similar to that of Lemma \ref{binomkeyobs}. Again
we may work in the affine toric variety $U_\rho$. Let $\gamma\in
H_\rho\setminus(-H_\rho)$. For any $\xi\in U_\rho(\R)$, we have
$\xi^\gamma=0$. Since $U$ is relatively compact, there exists $c>1$
with $c^{-1}\le\xi^\gamma\le c$ for all $\xi\in U$. We have
$\x^\gamma(\lambda_v(s)\xi)=s^{\bil\gamma v}\cdot\xi^\gamma$ for
$s>0$, and we conclude $\bil\gamma v>0$.
Thus $\bil\gamma v>0$ holds for every $\gamma\in H_\rho\setminus
(-H_\rho)$. This means $M\cap(-\R_\plus v)^*\cap\rho^*\subset
-\rho^*$, or $-\rho\subset\rho-\R_\plus v$ after dualizing. As before,
this implies $v\in\relint(\rho)$.
\end{proof}

Similarly to Proposition \ref{bdpolbinom}, we deduce:

\begin{cor}\label{bdpoltent}%
Let $U\ne\emptyset$ be an open and relatively compact subset of $Q$,
let $S=S_v(U)$ be the associated $v$-tentacle. Let $\sigma$ be a
pointed cone in $N_\R$. The ring of polynomials on $U_\sigma$ that
are bounded on $S$ is $B_{U_\sigma}(S)=\R[H]$ where $H=M\cap\sigma^*
\cap(\R_\plus v)^*$.
\qed
\end{cor}

\begin{lab}
Thus a polynomial function $f\in\R[U_\sigma]$ is bounded on $S$ if
and only if every monomial $\x^\alpha$ occuring in $f$ satisfies
$\bil\alpha v\ge0$. In particular, $B_{U_\sigma}(S)=\R$ is equivalent
to $\sigma+\R_\plus v=N_\R$.
\end{lab}


\section{Iitaka dimension on toric surfaces}\label{sec:trdeg}

Let $X$ be a non-singular projective surface over a field $k$. We
always assume that $X$ is absolutely irreducible. We first discuss
how the intersection matrix $A_D$ of an effective divisor $D$ on $X$
relates to the Iitaka dimension $\kappa(D)$ of $D$. Since $\kappa(D)$
is the transcendence degree of $\scrO(X\setminus D)$, these facts
have implications for rings of bounded polynomials on 2-dimensional
semi-algebraic sets, by Theorem \ref{ExistenceCompletionSurfaces}. In
general, the intersection matrix does not uniquely determine
$\kappa(D)$. However when $X$ is a toric surface and the divisor $D$
is toric, we show that $\kappa(D)$ can be read off from $A_D$ in a
simple manner (Proposition \ref{zsfsg}).

\begin{lab}
Given two divisors $D$, $D'$ on $X$, we denote by $D\mal D'$ the
intersection number of $D$ and $D'$. The intersection pairing is
invariant under linear equivalence and therefore induces a bilinear
pairing on the divisor class group $\Pic(X)$. As usual, we write
$D^2:=D\mal D$ for the self intersection number of $D$.
\end{lab}

\begin{dfn}
Let $D$ be an effective (not necessarily reduced) divisor on $X$
whose distinct irreducible components are $C_1,\dots,C_r$. We define
the \emph{intersection matrix} of $D$ to be the symmetric $r\times r$
matrix with integer entries $C_i\mal C_j$ ($i,\,j=1,\dots,r$), c.f.\
\cite{ii} 8.3. It will be denoted by $A_D$.
\end{dfn}

\begin{lab}
Let $D$ be an effective divisor on $X$. For $m\ge1$ let $\phi_m\colon
X\ratto|mD|$ be the rational map associated to the complete linear
series $|mD|$. The \emph{Iitaka dimension} of $D$ is defined to be
$$\kappa(X,D)\>:=\>\max_{m\ge1}\>\dim\phi_m(X),$$
see \cite{ii} Sect.~10.1 or \cite{la} 2.1.3.
It is well known that $\kappa(D,X)$ is equal to the transcendence
degree of $\scrO(X\setminus D)$, the ring of regular functions on the
open subvariety $X\setminus D:=X\setminus\supp(D)$ of $X$ (see
\cite{ii} Prop.~10.1).
The Iitaka dimension of $D$ is closely related to the intersection
matrix $A_D$:
\end{lab}

\begin{prop}\label{signintmat}%
\hfil
\begin{itemize}
\item[(a)]
If $A_D$ is negative definite, then $\kappa(X,D)=0$, i.e.\
$\scrO(X\setminus D)=k$.
\item[(b)]
If $D^2>0$ then $\kappa(X,D)=2$.
\end{itemize}
\end{prop}

\begin{proof}
(a) is \cite{ii} Proposition 8.5, (b) is Lemma 8.5.
Assertion (a) is also a consequence of Proposition \ref{negintmat}
below.
\end{proof}

\begin{cor}\label{posev2tr2}%
If $\scrO(X\setminus D)$ has transcendence degree $\le1$, then $A_D$
is negative semidefinite.
\end{cor}

In particular, $\kappa(X,D)$ is determined by the Sylvester signature
of $A_D$ whenever $A_D$ is non-singular.

\begin{proof}
Let $C_1,\dots,C_r$ be the irreducible components of $D$. The
intersection matrix $A_D$ has non-negative off-diagonal entries.
Therefore, if $A_D$ has a positive eigenvalue, there exist integers
$m_i\ge0$ with $(\sum_im_iC_i)^2>0$.
By Proposition \ref{signintmat}(b), this implies
$\trdeg\scrO(X\setminus D)=2$.
\end{proof}

Part (a) of \ref{signintmat} can be generalized as follows:

\begin{prop}\label{negintmat}%
Let $D\subset X$ be an effective divisor whose intersection matrix
$A_D$ is negative definite. Then for any line bundle $L$ on $X$, the
space $H^0(X\setminus D,\,\scrO(L))$ is a finite-dimensional
$k$-vector space.
\end{prop}

For the proof we need two lemmas.

\begin{lem}\label{rigid}%
Let $D$ be an effective divisor with irreducible components
$C_1,\dots,C_r$, and assume $C_i\mal D<0$ for $i=1,\dots,s$. Then for
any divisor $E$ there exists an integer $n_0=n_0(E)$ such that
$$|E+nD|\>=\>(n-n_0)D+|E+n_0D|$$
holds for all $n\ge n_0$.
\end{lem}

\begin{proof}
Say $D=\sum_{i=1}^rm_iC_i$, with $m_i\ge1$. Choose an integer $n$
such that the inequality
\begin{equation}\label{notigeungl}%
n\,(C_i\mal D)\><\>-C_i\mal\bigl(E+a_1C_1+\cdots+a_sC_s\bigr)
\end{equation}
holds for $i=1,\dots,r$ and every tuple $(a_1,\dots,a_r)$ with
$0\le a_j\le m_j$ ($j=1,\dots,r$). Then we claim
$$|E+(n+1)D|\>=\>|E+nD|+D.$$
Indeed, if $a_1,\dots,a_r$ are integers with $0\le a_j\le m_j$
($j=1,\dots,r)$, we show
$$\Bigl|E+nD+\sum_ja_jC_j\Bigr|\>=\>|E+nD|+\sum_ja_jC_j$$
by induction on $\sum_ja_j$. The assertion is trivial for $\sum_j
a_j=0$. If $(a_1,\dots,a_s)\ne(0,\dots,0)$ is a tuple with $0\le
a_j\le m_j$, and if $i$ is an index with $a_i\ge1$, we have
$$C_i\mal\Bigl(E+nD+\sum_ja_jC_j\Bigr)\><\>0$$
by \eqref{notigeungl}. Any effective divisor linearly equivalent to
$E+nD+\sum_ja_jC_j$ must therefore contain $C_i$, which implies
$$\Bigl|E+nD+\sum_ja_jC_j\Bigr|\>=\>\Bigl|-C_i+E+nD+\sum_ja_jC_j
\Bigr|+C_i,$$
and so
$$\Bigl|E+nD+\sum_ja_jC_j\Bigr|\>=\>|E+nD|+\sum_ja_jC_j$$
by the inductive hypothesis.
\end{proof}

\begin{lem}\label{posvector}%
If $x_1,\dots,x_r$ is a linear basis of $\R^r$, there exist integers
$m_1,\dots,m_r\ge1$ such that $x=\sum_im_ix_i$ satisfies $\bil x{x_i}
>0$ for all $i=1,\dots,r$.
\end{lem}

\begin{proof}
Let $K$ be the convex cone spanned by $x_1,\dots,x_r$, and let $K^*
=\{y\in\R^r$: $\bil xy\ge0\}$ be the dual cone. Since $x_1,\dots,x_r$
are a basis, both $K$ and $K^\ast$ have non-empty interior. We have
to show that
the interiors intersect. Assuming $\inter(K)\cap
\inter(K^*)=\emptyset$, there exists $0\ne z\in\R^r$ with $\bil xz
\ge0$ for all $x\in K$ and $\bil yz\le0$ for all $y\in K^*$. Hence
$z\in K^*\cap(-K^{**})=(K^*)\cap(-K)$, which implies $\bil zz\le0$,
whence $z=0$, a contradiction. Now $\interior(K)\cap\interior(K^\ast)$
is a non-empty open cone, hence it contains integer points with
respect to the basis $x_1,\dots,x_r$.
\end{proof}

\begin{proof}[Proof of Proposition \ref{negintmat}]
Let $C_1,\dots,C_r$ be the irreducible components of $D$, let $U:=
X\setminus D$, and let $E$ be a divisor on $X$ such that $L\cong
\scrO_X(E)$. Every section in $\Gamma(U,L)$ is a meromorphic section
of $L$ on $X$, which means that
$$\Gamma(U,L)\>=\>\bigcup_{n\ge1}\Gamma(X,E+nD)$$
(ascending union). Since $A_D$ is negative definite, we find integers
$m_1,\dots,m_r\ge1$ such that $D:=\sum_{i=1}^rm_iC_i$ satisfies
$C_i\mal D<0$ ($i=1,\dots,r$), using Lemma \ref{posvector}. By Lemma
\ref{rigid}, there exists $n_0\ge1$ such that $|E+nD|=(n-n_0)D+
|E+n_0D|$ for all $n\ge n_0$, which means $\Gamma(X,E+nD)=\Gamma
(X,E+n_0D)$. Hence $\Gamma(U,L)=\Gamma(X,E+n_0D)$, and so this space
has finite dimension.
\end{proof}

\begin{rem}
The hypothesis that $A_D$ is negative definite in Proposition
\ref{negintmat} entails $\scrO(X\setminus D)=k$.
One may wonder whether \ref{negintmat} remains true if only $\scrO
(X\setminus D)=k$ is assumed. An example due to Mondal and Netzer
\cite{mn} shows that this usually fails. We will revisit their
construction in Example \ref{Example:MondalNetzer} below. On the other
hand, we will see in \ref{zsfsg} below
that such problems do not occur in a toric setting.
\end{rem}

\begin{lab}\label{intform}%
Let $\Sigma$ be the fan of a non-singular projective toric surface
$X$. For $\rho\in\Sigma(1)$ let $Y_\rho=\ol{O_\rho}$. Let
$\rho_0,\dots,\rho_{m-1},\rho_m=\rho_0$ be the elements of
$\Sigma(1)$, written in cyclic order, so that $\rho_{i-1}$ and
$\rho_i$ bound a cone from $\Sigma(2)$ for $i=1,\dots,m$. Let $v_i$
be the primitive generator of $\rho_i$. The divisor class group of
$X$ is generated by the $Y_i=Y_{\rho_i}$, and the intersection form
on $X$ has the following description (see \cite[\S 10.4]{cls}):
Given $1\le i<m$, there is an integer $b_i$ such that $b_iv_i=
v_{i-1}+v_{i+1}$. Then we have $Y_i^2=-b_i$, $Y_i\mal Y_j=1$ if
$j-i=\pm1$, and $Y_i\mal Y_j=0$ otherwise. Similarly for $i=0$.
\end{lab}

\begin{lem}\label{schnittform}%
Let $n\ge1$, let $\rho_0,\dots,\rho_{n+1}$ be a sequence of pairwise
different cones in $\Sigma(1)$ such that $\rho_{i-1}$ and $\rho_i$
bound a cone from $\Sigma(2)$ for $i=1,\dots,n+1$. Let $l=\rho_0\cup
(-\rho_0)$, and let $A$ be the intersection matrix of the divisor
$\sum_{i=1}^nY_{\rho_i}$ on $X$. Then:
\begin{itemize}
\item[(a)]
$\det(A)=0$ \ $\iff$ \ $\rho_{n+1}=-\rho_0$;
\item[(b)]
$A\prec0$ \ $\iff$ \ $(-1)^n\,\det(A)>0$ \ $\iff$ \ $\rho_1$ and
$\rho_{n+1}$ lie (strictly) on the same side of the line~$l$;
\item[(c)]
$A$ is indefinite \ $\iff$ \ $(-1)^n\,\det(A)<0$ \ $\iff$ \ $\rho_1$
and $\rho_{n+1}$ lie (strictly) on opposite sides of~$l$.
\end{itemize}
In case (a) we have $A\preceq0$ and $\rk(A)=n-1$. In case (c) the
matrix $A$ has a unique positive eigenvalue.
\end{lem}

The hypothesis indicates that $\rho_0,\dots,\rho_{n+1}$ are given in
cyclic order and that there exists no further cone from $\Sigma(1)$
in between them. Since $\rho_{n+1}\ne\rho_0$, there exists at least
one cone in $\Sigma(2)$ that is not of the form $\cone(\rho_{i-1},\,
\rho_i)$ with $1\le i\le n+1$.

\begin{proof}
Let $v_i\in N$ be the primitive vector generating $\rho_i$, for
$i=0,\dots,n+1$. Let $b_1,\dots,b_n\in\Z$ be defined by $v_{i+1}+
v_{i-1}=b_iv_i$ ($i=1,\dots,n$). Then
$$A\ =\ \begin{pmatrix}-b_1&1\\1&-b_2&1\\&\ddots&\ddots&\ddots\\
&&1&-b_{n-1}&1\\&&&1&-b_n\end{pmatrix}$$
Let $\delta(b_1,\dots,b_i)$ be the upper left $i\times i$ principal
minor of $A$ ($i=1,\dots,n$). Then
$$v_{i+1}\>=\>(-1)^i\,\delta(b_1,\dots,b_i)\,v_1+(-1)^i\,
\delta(b_2,\dots,b_i)\,v_0$$
holds for $i=1,\dots,n$. In particular,
$$v_{n+1}\>=\>(-1)^n\det(A)\,v_1+(-1)^n\,\delta(b_2,\dots,b_n)\,
v_0.$$
Since $v_0,\,v_1$ are linearly independent and $v_0\ne v_{n+1}$, the
lemma follows easily from these identities.
\end{proof}

\begin{lab}\label{vorbereit}%
We keep the previous hypotheses. Let $T$ be a subset of $\Sigma(1)$,
let $T'=\Sigma(1)\setminus T$, and let
$$U\>=\>X\setminus\bigcup_{\tau\in T}Y_\tau,$$
an open toric subvariety of $X$. Let $C=\cone(\tau'\colon\tau'\in
T')\subset N_\R$, then $\scrO(U)=k[M\cap C^*]$, the semigroup
algebra of $M\cap C^*$.
The following list exhausts all possible cases:
\begin{itemize}
\item[1.]
$C=N_\R$. Then $|T'|\ge3$ and $\scrO(U)=k$.
\item[2.]
$C$ is a half-plane. Then $|T'|\ge3$ and $\scrO(U)\cong k[u]$
(polynomial ring in one variable).
\item[3.]
$C$ is a line. Then $|T'|=2$ and $\scrO(U)\cong k[u,u^{-1}]$ (ring
of Laurent polynomials in one variable).
\item[4.]
$C\setminus\{0\}$ is contained in an open half-plane. Then we have
$\trdeg\scrO(U)=2$.
\end{itemize}
\end{lab}

The following result shows that, for toric divisors on non-singular
toric surfaces, the Iitaka dimension is characterized by the
signature of the intersection matrix.

\begin{prop}\label{zsfsg}%
Let $X$ be a non-singular toric projective surface with fan $\Sigma$.
Let $T\subset\Sigma(1)$ with $T\ne\emptyset$, let $U=X\setminus
\bigcup_{\tau\in T}Y_\tau$, and let $A$ be the intersection matrix
of~$T$.
\begin{itemize}
\item[(a)]
$A\prec0$ \ $\iff$ \ $\scrO(U)=k$.
\item[(b)]
$A\preceq0$ and $\det(A)=0$ \ $\iff$ \ $\trdeg\scrO(U)=1$.
\item[(c)]
$A$ has a positive eigenvalue \ $\iff$ \ $\trdeg\scrO(U)=2$.
\end{itemize}
Moreover, in case (c) we have $\det(A)=0$ if and only if $|\Sigma(1)
\setminus T|\le1$.
\end{prop}

\begin{proof}
Write $T'=\Sigma(1)\setminus T$ as before.
The implications ``$\To$'' in (a) resp.\ in (c) hold for general
reasons (and could be easily reproved using \ref{vorbereit}), see
Proposition \ref{negintmat} and Corollary \ref{posev2tr2}.
The group $\Pic(X)$ is free abelian of rank $|\Sigma(1)|-2$. The
intersection form on $\Pic(X)$ is non-degenerate, and all of its
eigenvalues are negative except for one, by the Hodge index theorem.
Therefore it is clear for $|T'|\le1$ that $A$ is singular and has a
unique positive eigenvalue.
Also, $\trdeg\scrO(U)=2$ is clear for $|T'|\le1$, see
\ref{vorbereit}.

So we assume $|T'|\ge2$ for the rest of the proof.
The matrix $A$ is a block diagonal sum of matrices $A_1,\dots,A_r$.
For each $i=1,\dots,r$ there is a sequence $\rho_0,\dots,\rho_{n+1}$
as in Lemma \ref{schnittform} with $\rho_0,\,\rho_{n+1}\in T'$ and
$\rho_1,\dots,\rho_n\in T$, such that $A_i$ is the intersection
matrix of $\rho_1,\dots,\rho_n$. By \ref{vorbereit} we have
$\trdeg\scrO(U)=2$ if and only if all $\tau'\in T'$ are contained in
a common open half-plane (together with~$0$). By \ref{schnittform} it
is equivalent that $A_i$ is indefinite for one
index $i\in\{1,\dots,r\}$. Since $|T'|\ge2$, it is equivalent that
$A$ is indefinite and $\det(A)\ne0$. On the other hand, $\scrO(U)=k$
is by \ref{vorbereit} equivalent to the condition that $T'$ is not
contained in a half-plane. By \ref{schnittform}, this in turn is
equivalent to $A_i\prec0$ for every $i=1,\dots,r$, and hence to
$A\prec0$. This proves (a) and (c) together with the last statement,
and so it also implies (b).
\end{proof}

\begin{example}
Recalling the setup in \ref{abstractsetup}, we let $k=\R$, fix a
pointed cone $\sigma\in\Sigma$ and regard $X_\Sigma$ as a completion
of the affine toric variety $U_\sigma$. If $S\subset U_\sigma$ is an
open semi-algebraic set and the toric compatibility condition $(TC)$
is satisfied, the subset $T\subset\Sigma(1)$ in \ref{vorbereit}
consists of those $\tau\in\Sigma(1)$ with $\tau\nsubseteq\sigma$ and
$K_0(S)\cap\relint(\tau)=\emptyset$.

All of the four different cases in \ref{vorbereit} can occur in this
situation. Cases (1) and (4) may arise in a trivial way: For
example, let $\Sigma$ be the standard fan of $\P^2$ and
$\sigma=\{\cone(e_1,e_2)\}$, so that $U_\sigma\into X_\Sigma$ is the
usual embedding of the affine into the projective plane. Then if
$S=(\R^*)^2$, we find $T=\emptyset$ and $T'=\Sigma(1)$, hence 
$|T'|=3$  and $U=\P^2$, so that $\scrO(U)=B_{U_\sigma}(S)=\R$. The 
same will hold whenever $S\subset(\R^*)^2$ contains a non-empty open 
cone. On the other hand, if $S\subset(\R^*)^2$ is bounded, we have
$T=\{\cone(-e_1-e_2)\}$ and $T'=\{\cone(e_1),\cone(e_2)\}$, so that
$U=U_\sigma$ and $\scrO(U)=B_{U_\sigma}=\R[x_1,x_2]$ (for a more
interesting example leading to case (4), see also 
\cite[Example 3.10]{ps}).

In Example \ref{examples:CompatibleCompletions}(2), the fan
$\Sigma$ is the refinement of the standard fan of $\P^2$ in which
$\Sigma(1)=\{\cone(e_1),\cone(e_2),\cone(-e_1-e_2),
\pm\cone(e_1-ke_2)\}$. 
Here, $T=\{\cone(-e_1-e_2)\}$ and $C=\cone(\tau'\colon\tau'\in
T')=\cone(\pm(e_1-ke_2))$ is a half-plane, so we are in case
(2). Indeed we found $\scrO(U)=B_{U_\sigma}(S)=\R[x_1^kx_2]$.

Case (3) obviously will not come up if we start from
$\sigma=\cone(e_1,e_2)$, since $\R[x_1,x_2]$ cannot contain a ring of
Laurent polynomials. But if we consider for example
$S=\{(\xi_1,\xi_2)\in(\R^\ast)^2\colon 1<\xi_1<2\}$ and
$\Sigma$ as above with $k=0$, we have the same completion as in
Example \ref{examples:CompatibleCompletions}(1) but with
$\sigma=\{0\}$. Here, we find $T'=\{\pm\cone(e_2)\}$ and 
$\scrO(U)=B_{U_\sigma}(S)=\R[x_1,x_1^{-1}]$.
\end{example}


\section{Filtration by degree of boundedness}\label{sec:Filtrations}

\begin{lab}\label{mnfilt}%
Let $S\subset\R^d$ be a semi-algebraic set. In \cite{mn}, Mondal and
Netzer studied the following filtration on the polynomial ring. For
$n\ge0$ let
$$B_n(S)\>=\>\bigl\{f\in\R[x]\colon \exists\,g\in\R[x]\text{ with }
\deg(g)\le2n\text{ and }f^2\le g\text{ on }S\bigr\}.$$
Then the $B_n(S)$ form an ascending filtration on
$\R[x]=\R[x_1,\dots,x_d]$ by linear subspaces, satisfying $B_m(S)\,
B_n(S)\subset B_{m+n}(S)$ for alle $m,n\ge0$. Clearly $B_0(S)=B(S)$,
the ring of polynomials bounded on~$S$.
\end{lab}

\begin{lab}\label{lab:GeneralFiltration}
We propose to generalize the construction from \ref{mnfilt}.
Let $V$ be a normal affine $\R$-variety, let $S\subset
V(\R)$ be a semi-algebraic set, and let $V\subset X$ be an open dense
embedding into a normal and complete variety~$X$. We assume that the
completion is compatible with $S$, in the sense of
\ref{DefCompatibleCompletion}. Let $Y$ (resp.~$Y'$) be the union of
those irreducible components $Z$ of $X\setminus V$ for which
$\ol S\cap Z(\R)$ is empty (resp., non-empty), and put
$U=X\setminus Y$. Then $V=X\setminus(Y\cup Y')$. For $n\ge0$ let
$$L_{X,n}(S)\>=\>\Gamma(U,\scrO_X(nY')).$$
Since $Y'$ is disjoint from $V\subset U$, we may consider
$L_{X,n}(S)$ as a subspace of $\scrO_X(V)=\R[V]$, namely
$$L_{X,n}(S) = \{f\in\R[V]\colon\ord_Z(f)\ge -n\text{ for all
components }Z\text{ of }Y'\}.$$
The $L_{X,n}(S)$ ($n\ge0$) define an ascending and exhaustive
filtration of $\R[V]$ by linear subspaces, satisfying
$$L_{X,m}(S)\,L_{X,n}(S)\>\subset\>L_{X,m+n}(S)$$
for $m,n\ge0$. Moreover $L_{X,0}(S)=B_V(S)$ by Theorem
\ref{CompatibleCompletionIsoBounded}. In particular, the $L_{X,n}(S)$
are modules over the ring $B_V(S)$.

For $V=\A^d$ the affine space, the two filtrations $\{B_n(S)\}$ and
$\{L_{X,n}(S)\}$ on $\R[\x]$ are compatible in the following sense:
\end{lab}

\begin{prop}\label{Prop:CompareCompletions}
With notation as above, fix $m\ge 0$.
\begin{itemize}
\item[(a)]
There exists $n\ge 0$ such that $B_m(S)\subset L_{X,n}(S)$.
\item[(b)]
There exists $n\ge 0$ such that $L_{X,m}(S)\subset B_n(S)$.
\end{itemize}
\end{prop}

\begin{proof}
(a)
Choose $n$ such that $\R[x]_{2m}\subset L_{X,2n}(S)$ (the existence
of such $n$ is clear since $\R[x]_{2m}$ is finite-dimensional). Now
given $f\in B_m(S)$, choose $g\in\R[x]_{2m}$ such that $f^2\le g$ on
$S$. The rational function $f^2/(g+1)$ is defined and bounded on $S$.
We apply Theorem \ref{CompatibleCompletionIsoBounded} to the
$S$-compatible completion $(V\cap{\rm dom}(g+1))\subset X$ and
conclude that $\ord_Z(f^2/(g+1))\ge0$ for all components $Z$ of $Y'$.
Then $f\in L_{X,n}(S)$.

(b)
Choose $g\in\R[x]$ with $\ord_Z(g)\le -m$ for all components $Z$
of $Y'$. Let $f\in L_{X,m}(S)$. By Theorem
\ref{CompatibleCompletionIsoBounded}, the rational function
$f^2/(g^2+1)$ is defined and bounded on $S$. Thus if $n=\deg(g)$,
then $L_{X,m}\subset B_n(S)$.
\end{proof}

\begin{example}\label{Example:MondalNetzer}
  The following example is due to Mondal and
  Netzer \cite{mn}. Let $V=\A^2$, $\R[V]=\R[x,y]$, put
  \begin{IEEEeqnarray*}{rlll}
    f_1 &= x^3y+x^6-x,\quad &S_1 &= \{(a,b)\in\R^2\colon 2\ge
    f_1(a,b)\ge 1,\ a\ge 1\},\\
    f_2 &= x^3y-x^6-x,\quad &S_2 &= \{(a,b)\in\R^2\colon 2\ge
    f_2(a,b)\ge 1,\ a\ge 1\},
  \end{IEEEeqnarray*}
and let $S=S_1\cup S_2$.
  Applying the procedure of \ref{constrgoodcompl} to $S_1$,
  starting with
  $V\subset\P^2$, requires a sequence of nine blow-ups. In the
  resulting completion $V\subset X_1$, the complement
  $C_\infty=X_1\setminus V$ has ten irreducible components
  $E_0,\dots,E_9$, which are the strict transforms of the line
  $\P^2\setminus V$ and the exceptional divisors of the nine
  blow-ups. Only $E_9$ meets $\ol{S}$, so we
  need to consider the divisor $Y_1=\sum_{i=0}^8 E_i$. The
  configuration of the irreducible components of $Y_1$ is shown in its
  intersection graph
  \begin{center}
    \begin{tikzpicture}
      \node (e0) at (1,2) {$E_0$}; \node (e2) at (1,1) {$E_2$}; \node
      (e3) at (1,0) {$E_3$}; \node (e1) at (1,-1) {$E_1$}; \node (e4)
      at (2,0) {$E_4$}; \node (e5) at (3,0) {$E_5$}; \node (e6) at
      (4,0) {$E_6$}; \node (e7) at (5,0) {$E_7$}; \node (e8) at (6,0)
      {$E_8$};

      \foreach \from/\to in
      {e0/e2,e1/e3,e2/e3,e3/e4,e4/e5,e5/e6,e6/e7,e7/e8} \draw (\from)
      -- (\to);
    \end{tikzpicture}
  \end{center}
and the full intersection matrix is
\[
M_{Y_1}=
{\tiny
\begin{bmatrix}
 -1 & 0 & 1 & 0 & 0 & 0 & 0 & 0 & 0 \\
 0 & -3 & 0 & 1 & 0 & 0 & 0 & 0 & 0 \\
 1 & 0 & -2 & 1 & 0 & 0 & 0 & 0 & 0 \\
 0 & 1 & 1 & -2 & 1 & 0 & 0 & 0 & 0 \\
 0 & 0 & 0 & 1 & -2 & 1 & 0 & 0 & 0 \\
 0 & 0 & 0 & 0 & 1 & -2 & 1 & 0 & 0 \\
 0 & 0 & 0 & 0 & 0 & 1 & -2 & 1 & 0 \\
 0 & 0 & 0 & 0 & 0 & 0 & 1 & -2 & 1 \\
 0 & 0 & 0 & 0 & 0 & 0 & 0 & 1 & -2
\end{bmatrix}.
}
\]
This matrix has a postive eigenvalue, which shows that $B_V(S_1)$ has
transcendence degree $2$. The compatible completion for $S_2$ is
isomorphic to that of $S_1$.

When resolving the union $S=S_1\cup S_2$, the first three blow-ups are
the same as for $S_1$. The blow-ups for $S_1$ and $S_2$ in the fourth
step are centered at two different points of $E_3$. In the resulting
$S$-compatible completion $X$ of $V$, the divisor $Y$ has intersection
graph
  \begin{center}
    \begin{tikzpicture}
      \node (e0) at (0,2) {$E_0$};
      \node (e2) at (0,1) {$E_2$};
      \node (e3) at (0,0) {$E_3$};
      \node (e1) at (0,-1) {$E_1$};
      \node (e4) at (1,0) {$E_4$};
      \node (e5) at (2,0) {$E_5$};
      \node (e6) at (3,0) {$E_6$};
      \node (e7) at (4,0) {$E_7$};
      \node (e8) at (5,0) {$E_8$};
      \node (f4) at (-1,0) {$E'_4$};
      \node (f5) at (-2,0) {$E'_5$};
      \node (f6) at (-3,0) {$E'_6$};
      \node (f7) at (-4,0) {$E'_7$};
      \node (f8) at (-5,0) {$E'_8$};

      \foreach \from/\to in
      {e0/e2,e1/e3,e2/e3,e3/e4,e4/e5,e5/e6,e6/e7,e7/e8,e3/f4,f4/f5,f5/f6,f6/f7,f7/f8} \draw (\from)
      -- (\to);
    \end{tikzpicture}
  \end{center}
and intersection matrix
\[
M_Y=
{\tiny
\left[
\begin{array}{cccccccccccccc}
 -1 & 0 & 1 & 0 & 0 & 0 & 0 & 0 & 0 & 0 & 0 & 0 & 0 & 0 \\
 0 & -3 & 0 & 1 & 0 & 0 & 0 & 0 & 0 & 0 & 0 & 0 & 0 & 0 \\
 1 & 0 & -2 & 1 & 0 & 0 & 0 & 0 & 0 & 0 & 0 & 0 & 0 & 0 \\
 0 & 1 & 1 & -3 & 1 & 0 & 0 & 0 & 0 & 1 & 0 & 0 & 0 & 0 \\
 0 & 0 & 0 & 1 & -2 & 1 & 0 & 0 & 0 & 0 & 0 & 0 & 0 & 0 \\
 0 & 0 & 0 & 0 & 1 & -2 & 1 & 0 & 0 & 0 & 0 & 0 & 0 & 0 \\
 0 & 0 & 0 & 0 & 0 & 1 & -2 & 1 & 0 & 0 & 0 & 0 & 0 & 0 \\
 0 & 0 & 0 & 0 & 0 & 0 & 1 & -2 & 1 & 0 & 0 & 0 & 0 & 0 \\
 0 & 0 & 0 & 0 & 0 & 0 & 0 & 1 & -2 & 0 & 0 & 0 & 0 & 0 \\
 0 & 0 & 0 & 1 & 0 & 0 & 0 & 0 & 0 & -2 & 1 & 0 & 0 & 0 \\
 0 & 0 & 0 & 0 & 0 & 0 & 0 & 0 & 0 & 1 & -2 & 1 & 0 & 0 \\
 0 & 0 & 0 & 0 & 0 & 0 & 0 & 0 & 0 & 0 & 1 & -2 & 1 & 0 \\
 0 & 0 & 0 & 0 & 0 & 0 & 0 & 0 & 0 & 0 & 0 & 1 & -2 & 1 \\
 0 & 0 & 0 & 0 & 0 & 0 & 0 & 0 & 0 & 0 & 0 & 0 & 1 & -2
\end{array}
\right]},
\]
which is negative semidefinite of corank $1$. Mondal and Netzer show
through direct computation that $B(S)=\R$, and at the same time that
$B_1(S)$ has infinite dimension over $\R$. We thus conclude from
Proposition \ref{Prop:CompareCompletions} that there exists $n\ge 0$
such that $\dim H^0(X\setminus Y,\,\scrO(n(E_9+E_9')))=\infty$, even
though $\scrO_X(X\setminus Y)=B(S)=\R$.
\end{example}

We now show that the phenomenon in the above example does not occur
for semi-algebraic sets inside a compatible toric completion. It
follows in particular that the set constructed by Mondal and Netzer
does not admit such a completion.

\begin{prop}\label{globsecttoric}%
Let $U$ be a toric variety. For any Weil divisor $D$ on $U$, the
space $H^0(U,\scrO_U(D))$ is finitely generated as a module over
$\scrO(U)=H^0(U,\scrO_U)$.
\end{prop}

\begin{proof}
Let $\Sigma$ be the fan associated to $U$. For every $\rho\in
\Sigma(1)$ let $u_\rho\in N$ be the primitive generator of $\rho$. We
can assume that the Weil divisor $D$ is torus invariant (\cite{cls}
4.1.3).
So $D=\sum_{\rho\in\Sigma(1)}m_\rho Y_\rho$ with $m_\rho\in\Z$. By
\cite{cls} 4.1.2 and 4.3.2, the space $H^0(U,\scrO(D))$ is linearly
spanned by the characters $\x^\beta$ with $\beta$ in
$$B\>=\>\Bigl\{\beta\in M\colon\bil\beta{u_\rho}\ge-m_\rho
\text{ for all }\rho\in\Sigma(1)\Bigr\}.$$
On the other hand, the ring $\scrO(U)$ is linearly spanned by the
characters $\x^\alpha$ with $\alpha$ in
$$A\>=\>\Bigl\{\alpha\in M\colon\bil\beta{u_\rho}\ge0\text{ for all }
\rho\in\Sigma(1)\Bigr\}.$$
From Dickson's Lemma it follows that there exists a finite subset
$B_0$ of $B$ such that $B=B_0+A$.
Hence the $\scrO(U)$-module $\scrO(D)$ is generated by the characters
$\x^\beta$ with $\beta\in B_0$.
\end{proof}

\begin{cor}\label{corglobsecttoric}
Let $V$ be an affine toric $\R$-variety, let $S\subset V(\R)$ be a
semi-algebraic set. Assume that there exists a toric completion
$V\subset X$ of $V$ which is compatible with $S$. Then $L_{X,n}(S)$
is finitely generated as a module over $B_V(S)$, for every $n\ge0$.
\end{cor}

Note that Proposition \ref{abstrfw} provides a sufficient condition
for the existence of a completion $X$ as required. In
particular, such $X$ exists if $S$ is defined by binomial
inequalities (\ref{bdpolbinom})

\begin{proof}
By assumption, $X$ is a complete toric variety, and every irreducible
component of $X\setminus V$ is torus invariant. Let $Y$ (resp.\ $Y'$)
be the union of those irreducible components of $X\setminus V$ for
which $\ol S\cap Y(\R)=\emptyset$ (resp.\ $\ol S\cap Y(\R)\ne
\emptyset$), and write $U=X\setminus Y$. Then $U$ is a toric variety.
By definition, $L_{X,n}(S)=H^0(U,\,\scrO(-nY'))$ for $n\ge0$. By
Proposition \ref{globsecttoric}, $L_{X,n}(S)$ is finitely generated
as a module over $\scrO(U)$, and $\scrO(U)=B_V(S)$ by Theorem
\ref{CompatibleCompletionIsoBounded}.
\end{proof}

In particular, $B_V(S)=\R$ implies that the spaces $L_{X,n}(S)$ are
all finite-dimen\-sional. If $V=\A^d$, this also implies that the
spaces $B_n(S)$ of Mondal-Netzer are all finite-dimensional,
using Proposition \ref{Prop:CompareCompletions}.


\section{Positive polynomials and stability}\label{sec:PositivePolynomials}

Let $V$ be an irreducible affine $\R$-variety, $S\subset V(\R)$ a
closed \sa\ set and
\[
\scrP_V(S) = \bigl\{f\in\R[V]\colon f|_S\ge 0\bigr\},
\]
the cone of non-negative regular functions on $S$. We will always
assume that $S$ is the closure of its interior in $V(\R)$. (This
property is sometimes referred to by saying that $S$ is ``regular''.)

\begin{dfn}
We say that $\scrP_V(S)$ is \emph{totally stable} if the following
holds: For every finite-dimensional subspace $U$ of $\R[V]$ there
exists a finite-dimensional subspace $W$ of $\R[V]$ such that for all
$r\ge 2$ and $f_1,\dots,f_r\in\scrP_V(S)$:
\[
 f_1+\cdots+f_r\in U\quad\Longrightarrow\quad f_1,\dots,f_r\in W.
\]
\end{dfn}

Consider the case $V=\A^n_\R$, $\R[V]=\R[x]$ with $x=(x_1,\dots,x_n)$.
Then $\scrP(S)$ is totally stable if and only if for every $d\ge 0$
there exists $e\ge d$ such that whenever $f_1,\dots,f_r\in\scrP(S)$
are such that $\deg(\sum_{i=1}^rf_i)\le d$, it follows that
$\deg(f_i)\le e$ for all $i$. Note that if $S=\R^n$, the leading forms
of two non-negative polynomials cannot cancel, so we may take $e=d$.

The property of total stability has consequences for the existence of
degree bounds for representations of positive polynomials by weighted
sums of squares, and to the moment problem in polynomial optimisation
and functional analysis. These questions have been a major motivation
for the study of bounded polynomials. The precise statement is as
follows: The moment problem for $S$ is said to be finitely solvable if
$\scrP_V(S)$ contains a dense finitely generated quadratic module $M$
(where dense means that no element of $\scrP_V(S)$ can be strictly
separated from $M$ by any linear functional on $\R[V]$). The main
result of \cite{sc} implies the following.

\begin{thm}
  If $S\subset V(\R)$ has dimension at least $2$ and $\scrP_V(S)$ is totally
  stable, then the moment problem for $S$ is not finitely solvable.\qed
\end{thm}

See \cite{ma}, \cite{pws} and \cite{sc} for a fuller discussion and
further references.

\bigskip We now examine total stability for open \sa\ sets that
admit a compatible completion. First note the following simple
observation that holds without any additional assumptions.

\begin{prop}\label{prop:BSRTotallyStable}
  If $\scrP_V(S)$ is totally stable, then $B_V(S)=\R$.
\end{prop}

\begin{proof}
  Let $f\in B_V(S)$, say $|f|\le\lambda$ on $S$ for some $\lambda\in\R$.
  Then $f^{2n}, \lambda^{2n}-f^{2n}\in\scrP_V(S)$ for all $n\ge
  1$.
  Since the sum of these two elements is constant, total stability
  implies the existence of a finite-dimensional subspace $W$ of
  $\R[V]$ containing $f^{2n}$ for all $n\ge 1$. Thus $f$ is
  algebraic over $\R$ and therefore constant.
\end{proof}

We are interested in the converse. Suppose that $V$ is normal and
admits an $S$-compatible completion
$V\into X$. Following the notation in \ref{lab:GeneralFiltration}, let
$Y$ be the union of all irreducible components $Z$ of $X\setminus V$
for which $\ol{S}\cap Z(\R)$ is empty, and let $Y'$ the union of those
for which $\ol{S}\cap Z(\R)$ is dense in $Z$. Consider the filtration
\[
L_{X,1}(S)\subset L_{X,2}(S)\subset\cdots
\]
of $\R[V]$ by (not necessarily finite-dimensional) subspaces, defined
in \ref{lab:GeneralFiltration}.
Given $f_1,\dots,f_r\in\scrP_V(S)$, we have
\[
\ord_Z(f_1+\cdots+f_r) = \min\{\ord_Z(f_1),\dots,\ord_Z(f_r)\}
\]
for every component $Z$ of $Y'$, by \cite[Lemma
3.4]{ps}. Thus if $f_1+\cdots+f_r\in L_{X,n}(S)$, then already
$f_1,\dots,f_r\in L_{X,n}(S)$. It follows that if the spaces $L_{X,n}(S)$ are
finite-dimensional, then $\scrP_V(S)$ is totally stable.

In the two-dimensional case, we have a sufficient condition in terms of
the intersection matrix, as seen in Section \ref{sec:trdeg}.

\begin{prop}\label{totstabnegintmat}%
  Let $V$ be a non-singular real affine surface and $S\subset V(\R)$
  a closed \sa\ set that is the closure of its interior. Assume that
  $V$ admits a non-singular
  $S$-compatible completion $V\into X$ such that the intersection
  matrix of the divisor $Y$ defined as above is negative
  definite. Then $\scrP_V(S)$ is totally stable.
\end{prop}

\begin{proof}
  In this case, the spaces $L_{X,n}(S)$ are finite-dimensional by
  Prop.~\ref{negintmat}, so that $\scrP_V(S)$ is totally stable by the
  preceding discussion.
\end{proof}

In view of Prop.~\ref{prop:BSRTotallyStable}, we are also led to the
question of whether $B_V(S)=\R$ implies that
the spaces $L_{X,n}(S)$ are finite-dimensional over $\R$ for all $n\ge
0$. The example of Mondal and Netzer discussed in \ref{mnfilt} shows
this to be false in general. On the other hand, it is true in the
toric setting, as we saw in the preceding section.

\begin{thm}\label{Thm:TotalStabilityToric}
  Let $S$ be a closed \sa\ set in an affine toric variety $V$ that is
  the closure of its interior. Assume that
  $B_V(S)=\R$ and that $V$ admits a toric $S$-compatible completion
  $V\into X$. Then $\scrP_V(S)$ is totally stable.
\end{thm}

\begin{proof}
  By Corollary \ref{corglobsecttoric}, the spaces $L_{X,n}(S)$ constructed
  from the completion $X$ are finite-dimensional over
  $B_V(S)=\R$. Hence $\scrP_V(S)$ is totally stable by the
  argument above.
\end{proof}


\end{document}